\documentclass[12pt]{amsart}

\title[Not large in HOD]{Large cardinals need not be large in HOD}

\author[Y.~Cheng]{Yong Cheng}
 \address[Y.~Cheng]
         {Department of Philosophy, Wuhan University, BaYi Road 299, Wuchang District, Wuhan, Hubei Province, P.R.China. 430072}
 \email{world-cyr@hotmail.com}

\author[S.-D.~Friedman]{Sy-David Friedman}
\address[S.-D. Friedman]{Kurt G\"odel Research Center,
                         University of Vienna,
                         W\"ahringer Strasse 25 A-1090 Vienna, Austria}
\email{sdf@logic.univie.ac.at}
\urladdr{http://www.logic.univie.ac.at/~sdf/}

\author[Hamkins]{Joel David Hamkins}
 \address[J.~D.~Hamkins]
         {Philosophy, New York University \&
         Mathematics, Philosophy, Computer Science, The Graduate Center of The City University of New York \&
         Mathematics, College of Staten Island of CUNY}
 \email{jhamkins@gc.cuny.edu}
 \urladdr{http://jdh.hamkins.org}

\thanks{The second author wishes to thank the Austrian Science Fund (FWF) for its support through project \# P23316-N13. The third author is thankful for the support of his research provided by Simons Foundation grant 209252, PSC-CUNY grant 66563-00 44 and grant 80209-06 20 from the CUNY Collaborative Incentive Research Award program. Commentary concerning this paper can be made at \url{http://jdh.hamkins.org/large-cardinals-need-not-be-large-in-hod}.}

\usepackage{amssymb,bbm}
\usepackage[hidelinks]{hyperref}
\newtheorem{theorem}{Theorem}

\newtheorem{corollary}[theorem]{Corollary}

\newtheorem{lemma}[theorem]{Lemma}

\newtheorem{question}[theorem]{Question}

\newtheorem*{questions*}{Questions}
\newtheorem*{mainquestion*}{Main Question} 
\newtheorem{observation}[theorem]{Observation}

\newtheorem{conjecture}[theorem]{Conjecture}

\newcommand{\QED}{\end{proof}}

\def\proclaim[#1]{{\bf #1}}
\def\BF#1.{{\bf #1.}}

\newcommand{\Erdos}{Erd\H{o}s}

\renewcommand{\P}{{\mathbb P}}
\newcommand{\Q}{{\mathbb Q}}

\newcommand{\R}{{\mathbb R}}

\newcommand{\Gtail}{{G_{\!\scriptscriptstyle\rm tail}}}

\newcommand{\Ptail}{{\P_{\!\scriptscriptstyle\rm tail}}}

\newcommand{\Mbar}{{\overline{M}}}

\newcommand{\Vbar}{{\overline{V}}}

\newcommand{\Rdot}{{\dot\R}}
\newcommand{\Tdot}{{\dot T}}

\newcommand{\one}{\mathbbm{1}} 
\newcommand{\of}{\subseteq}

\newcommand{\set}[1]{\{\,{#1}\,\}}

\newcommand{\elesub}{\prec}

\newcommand{\Add}{\mathop{\rm Add}}

\newcommand{\restrict}{\upharpoonright} 
\newcommand{\satisfies}{\models}
\newcommand{\forces}{\Vdash}

\newcommand{\Union}{\bigcup}

\newcommand{\intersect}{\cap}

\newcommand{\smalllt}{\mathrel{\mathchoice{\raise2pt\hbox{$\scriptstyle<$}}{\raise1pt\hbox{$\scriptstyle<$}}{\raise0pt\hbox{$\scriptscriptstyle<$}}{\scriptscriptstyle<}}}
\newcommand{\smallleq}{\mathrel{\mathchoice{\raise2pt\hbox{$\scriptstyle\leq$}}{\raise1pt\hbox{$\scriptstyle\leq$}}{\raise1pt\hbox{$\scriptscriptstyle\leq$}}{\scriptscriptstyle\leq}}}
\newcommand{\lt}{\smalllt}

\newcommand{\ltkappa}{{{\smalllt}\kappa}}
\newcommand{\leqkappa}{{{\smallleq}\kappa}}

\newcommand{\leqtheta}{{{\smallleq}\theta}}

\newcommand{\boolval}[1]{\mathopen{\lbrack\!\lbrack}\,#1\,\mathclose{\rbrack\!\rbrack}}
\def\[#1]{\boolval{#1}}
\newbox\gnBoxA
\newdimen\gnCornerHgt
\setbox\gnBoxA=\hbox{\tiny$\ulcorner$}
\global\gnCornerHgt=\ht\gnBoxA
\newdimen\gnArgHgt
\def\gcode #1{%
\setbox\gnBoxA=\hbox{$#1$}%
\gnArgHgt=\ht\gnBoxA%
\ifnum     \gnArgHgt<\gnCornerHgt \gnArgHgt=0pt%
\else \advance \gnArgHgt by -\gnCornerHgt%
\fi \raise\gnArgHgt\hbox{\tiny$\ulcorner$} \box\gnBoxA %
\raise\gnArgHgt\hbox{\tiny$\urcorner$}}
\newcommand{\UnderTilde}[1]{{\setbox1=\hbox{$#1$}\baselineskip=0pt\vtop{\hbox{$#1$}\hbox to\wd1{\hfil$\sim$\hfil}}}{}}
\newcommand{\Undertilde}[1]{{\setbox1=\hbox{$#1$}\baselineskip=0pt\vtop{\hbox{$#1$}\hbox to\wd1{\hfil$\scriptstyle\sim$\hfil}}}{}}
\newcommand{\undertilde}[1]{{\setbox1=\hbox{$#1$}\baselineskip=0pt\vtop{\hbox{$#1$}\hbox to\wd1{\hfil$\scriptscriptstyle\sim$\hfil}}}{}}
\newcommand{\UnderdTilde}[1]{{\setbox1=\hbox{$#1$}\baselineskip=0pt\vtop{\hbox{$#1$}\hbox to\wd1{\hfil$\approx$\hfil}}}{}}
\newcommand{\Underdtilde}[1]{{\setbox1=\hbox{$#1$}\baselineskip=0pt\vtop{\hbox{$#1$}\hbox to\wd1{\hfil\scriptsize$\approx$\hfil}}}{}}

\newcommand{\st}{\mid}

\renewcommand{\iff}{\mathrel{\leftrightarrow}}
\newcommand{\Iff}{\mathrel{\Leftrightarrow}}

\newcommand{\iso}{\cong}
\def\<#1>{\left\langle#1\right\rangle}

\newcommand{\Ord}{\mathop{{\rm Ord}}}

\newcommand{\GCH}{{\rm GCH}}
\newcommand{\SCH}{{\rm SCH}}

\newcommand{\HOD}{{\rm HOD}}

\newcommand{\cell}[1]{\boxit{\hbox to 17pt{\strut\hfil$#1$\hfil}}}
\newcommand{\head}[2]{\lower2pt\vbox{\hbox{\strut\footnotesize\it\hskip3pt#2}\boxit{\cell#1}}}
\newcommand{\boxit}[1]{\setbox4=\hbox{\kern2pt#1\kern2pt}\hbox{\vrule\vbox{\hrule\kern2pt\box4\kern2pt\hrule}\vrule}}
\newcommand{\Col}[3]{\hbox{\vbox{\baselineskip=0pt\parskip=0pt\cell#1\cell#2\cell#3}}}
\newcommand{\tapenames}{\raise 5pt\vbox to .7in{\hbox to .8in{\it\hfill input: \strut}\vfill\hbox to
.8in{\it\hfill scratch: \strut}\vfill\hbox to .8in{\it\hfill output: \strut}}}
\newcommand{\Head}[4]{\lower2pt\vbox{\hbox to25pt{\strut\footnotesize\it\hfill#4\hfill}\boxit{\Col#1#2#3}}}
\newcommand{\Dots}{\raise 5pt\vbox to .7in{\hbox{\ $\cdots$\strut}\vfill\hbox{\ $\cdots$\strut}\vfill\hbox{\
$\cdots$\strut}}}
\newcommand{\df}{\it} 
\begin{document}

\begin{abstract}
We prove that large cardinals need not generally exhibit their large cardinal nature in \HOD. For example, a supercompact cardinal $\kappa$ need not be weakly compact in \HOD, and there can be a proper class of supercompact cardinals in $V$, none of them weakly compact in $\HOD$, with no supercompact cardinals in $\HOD$. Similar results hold for many other types of large cardinals, such as measurable and  strong cardinals.
\end{abstract}

\maketitle

\section{Introduction}

In this article, we shall prove that large cardinals need not generally exhibit their large cardinal nature in \HOD, the inner model of hereditarily ordinal-definable sets, and there can be a divergence in strength between the large cardinals of the ambient set-theoretic universe $V$ and those of $\HOD$. Our general theme concerns the questions:
\begin{question}\ \label{Question.HODquestion}
 \begin{enumerate}
  \item To what extent must a large cardinal in $V$ exhibit its large cardinal properties in \HOD?
  \item To what extent does the existence of large cardinals in $V$ imply the existence of large cardinals in $\HOD$?
 \end{enumerate}
\end{question}
For large cardinal concepts beyond the weakest notions, we prove, the answers are generally negative. In Theorem~\ref{Theorem.SupercompactNotWCinHOD}, for example, we construct a model with a supercompact cardinal that is not weakly compact in $\HOD$, and Theorem~\ref{Theorem.ProperClassSCNotLargeInHOD} extends this to a proper class of supercompact cardinals, none of which is weakly compact in $\HOD$, thereby providing some strongly negative instances of (1). The same model has a proper class of supercompact cardinals, but no supercompact cardinals in $\HOD$, providing a negative instance of (2). A natural strengthening of these situations would be a model with a proper class of supercompact cardinals, but no weakly compact cardinals in $\HOD$, but such a situation is impossible for reasons we discuss in Section~\ref{Section.Limitations}. Conjecture~\ref{Conjecture.Woodin} is an intriguing positive instance of (2) recently proposed by W. Hugh Woodin, and many other natural possibilities remain as open questions.

\section{Large cardinals that are not weakly compact in \HOD}\label{Section.IndividualLC}

Let us begin with an elementary case, showing that a measurable cardinal $\kappa$ need not be measurable in $\HOD$. We shall subsequently strengthen this in Theorem~\ref{Theorem.SupercompactNotWCinHOD} to show that even a supercompact cardinal need not be weakly compact in \HOD.

\begin{theorem}\label{Theorem.MeasurableNotMeasurableInHOD}
If $\kappa$ is measurable, then there is a forcing extension in which $\kappa$ is measurable, but not measurable in \HOD.
\end{theorem}

\begin{proof}
Suppose that $\kappa$ is measurable in $V$, and assume without loss of generality, by forcing if necessary \cite{Jensen1974:MeasurableCardinalsAndTheGCH}, that the \GCH\ holds. Let $\P$ be the Easton-support $\kappa$-iteration, which forces with $\Q_\gamma=\Add(\gamma,1)$ at every inaccessible cardinal $\gamma\lt\kappa$. Suppose that $G\of\P$ is $V$-generic. The forcing is $\kappa$-c.c. and preserves the inaccessibility of $\kappa$. There is no forcing (yet) at stage $\kappa$.

We show first that $\kappa$ is not measurable in $V[G]$. If it were measurable there, then let $j:V[G]\to \Mbar$ be a normal ultrapower embedding. By elementarity, we may decompose $\Mbar=M[j(G)]$ as a forcing extension of the ground model $M=\Union_\alpha j(V_\alpha)$. Since $\kappa$ is inaccessible in $M$, there is nontrivial forcing at stage $\kappa$ in $j(\P)$, and so $j(G)=G*g*\Gtail$ for some $M[G]$-generic filter $g$ for $\Add(\kappa,1)^{M[G]}$. But since $P(\kappa)^V\of M$ and $\P$-names for subsets of $\kappa$ are coded in $H_{\kappa^+}$, it follows that $P(\kappa)^{V[G]}\of M[G]$. In particular, it must be that $g\in M[G]$, contrary to the assumption that $g$ is $M[G]$-generic.

Next, we perform additional forcing over $V[G]$. Let $\R$ be the set-support product forcing to code $P(\kappa)^{V[G]}$ into the \GCH\ or  $\Diamond^*_\lambda$ patterns (see Appendix) at regular cardinals above $\kappa^{++}$, with $\leqkappa$-closed forcing also preserving $P(P(\kappa))$. Suppose that $H\of\R$ is $V[G]$-generic, and then lastly let $g\of\kappa$ be $V[G][H]$-generic for $\Add(\kappa,1)$, which is the same in $V[G][H]$ as in $V[G]$. Our final model is $V[G][H][g]$.

We claim that $\kappa$ is not measurable in $\HOD^{V[G][H][g]}$. To see this, note first that because $\kappa$ is not measurable in $V[G]$ and $H$ preserves $P(P(\kappa))$, it follows that $\kappa$ is not measurable in $V[G][H]$. Next, observe that since $P(\kappa)^{V[G]}$, which is the same as $P(\kappa)^{V[G][H]}$, is explicitly coded in $V[G][H]$ and this coding is preserved by the subsequent forcing to add $g$ (using Lemma~\ref{Lemma.Diamond*InvariantBySmallForcing} in the case of $\Diamond^*$ coding), it follows that $P(\kappa)^{V[G][H]}\of\HOD^{V[G][H][g]}$. Meanwhile, since $\Add(\kappa,1)$ is weakly homogeneous (see appendix for definition), it follows by Lemma~\ref{Lemma.WeaklyHomogenousControlsHOD} that $\HOD^{V[G][H][g]}\of V[G][H]$. So $\HOD^{V[G][H][g]}$ and $V[G][H]$ agree on $P(\kappa)$, but there is no measure on $\kappa$ in $V[G][H]$, and consequently there can be none in $\HOD^{V[G][H][g]}$, because this model is contained in $V[G][H]$.

Meanwhile, $\kappa$ is measurable in $V[G][H][g]$. The standard master-condition lifting arguments show that $\kappa$ remains measurable in $V[G][g]$, and this will be preserved to $V[G][g][H]=V[G][H][g]$. Specifically, fix a normal ultrapower embedding $j:V\to M$, and consider the forcing $j(\P)$, which factors as $\P*\Q_\kappa*\Ptail$. Using $2^\kappa=\kappa^+$, one may construct an $M[G][g]$-generic filter $\Gtail\of\Ptail$ and lift the embedding to $j:V[G]\to M[j(G)]$ with $j(G)=G*g*\Gtail$. Using $g$ as a master condition in $j(\Q_\kappa)$, one similarly constructs an $M[j(G)]$-generic filter $j(g)\of j(\Q_\kappa)$ and lifts the embedding fully to $j:V[G][g]\to M[j(G)][j(g)]$, witnessing that $\kappa$ is measurable in $V[G][g]$. Next, since $\R$ is $\leqkappa$-closed and $\Add(\kappa,1)$ is $\kappa^+$-c.c. in $V[G]$, it follows by Lemma~\ref{Lemma.ClosureDistributive} that $\R$ remains $\leqkappa$-distributive in $V[G][g]$. Thus, the forcing to add $H$ over $V[G][g]$ adds no new subsets of $\kappa$ not already in $V[G][g]$, and so $\kappa$ remains measurable in $V[G][g][H]$, which is the same as $V[G][H][g]$, as desired.

We have therefore produced a model $V[G][H][g]$ in which $\kappa$ is measurable, but not measurable in \HOD, as desired.
\end{proof}

In fact, we could easily have coded much more of $V[G]$ than just $P(\kappa)$. We could have coded some large $V_\theta[G]$, for example, or indeed, we could let $\R$ be the proper class $\leqkappa$-closed forcing to code all the sets of $V[G]$ into the \GCH\ or $\Diamond^*$ patterns, with the result by homogeneity that $\HOD^{V[G][H][g]}=V[G]$, where $\kappa$ is not measurable, though it is measurable in $V[G][H][g]$ as above. Meanwhile, let us now modify the proof in order to provide a stronger result:

\begin{theorem}
If $\kappa$ is a measurable cardinal, then there is a forcing extension in which $\kappa$ remains measurable, but is not weakly compact in $\HOD$.
\end{theorem}

\begin{proof}
Suppose that $\kappa$ is a measurable cardinal in $V$. By preparatory forcing, if necessary, we may assume that the measurability of $\kappa$ is indestructible by $\Add(\kappa,1)$. (This can be accomplished, for example, by first forcing $2^\kappa=\kappa^+$ and then performing the Easton-support iteration that adds a Cohen subset to every inaccessible cardinal up to and including $\kappa$; alternatively, one can use the lottery preparation \cite{Hamkins2000:LotteryPreparation}.) Consider now the forcing $\Add(\kappa,1)$ to add a Cohen subset to $\kappa$, and suppose that $g\of\kappa$ is the resulting $V$-generic Cohen subset. By Lemma~\ref{Lemma.Kunen}, the forcing $\Add(\kappa,1)$ may be factored as a two-step forcing iteration $\SS*\Tdot$, where the first step $\SS$ is the forcing to add a weakly homogeneous $\kappa$-Suslin tree $T$ and the second step $\Tdot$ simply forces with that tree, adding a branch through it. In our case, we may factor the extension $V[g]$ as $V[T][b]$, where $g\iso T*b$, first adding the weakly homogeneous $\kappa$-Suslin tree $T$ and then adding a branch $b$ through $T$. Since $T$ is a $\kappa$-Suslin tree in $V[T]$, it follows that $\kappa$ is not weakly compact there.

In $V[T]$, the tree $T$ has size $\kappa$ and so there is some set $E\of\kappa$ that codes $T$ in some canonical way. Let $\R$ be the Easton product forcing in $V[T]$ that codes $E$ into the \GCH\ or $\Diamond^*_\lambda$ patterns on the next $\kappa$ many regular cardinals above $\kappa$. Suppose that $H\of\R$ is $V[T]$-generic. This is equivalent to assuming $H$ is $V[g]$-generic, since $T$ is still $\kappa$-Suslin in $V[T][H]$, and so $b$ is $V[T][H]$-generic (as any cofinal branch through a Suslin tree is), making $b$ and $H$ mutually $V[T]$-generic. Thus, because $\R\times T\iso T\times\R$ in $V[T]$, we may view our final model as any of the three forcing iterations $V[T][H][b]=V[T][b][H]=V[g][H]$.

By our indestructibility assumption on $\kappa$, it follows that $\kappa$ is measurable in $V[g]$, and the forcing $\R$ is $\leqkappa$-closed in $V[T]$ and consequently remains $\leqkappa$-distributive in $V[g]$ by Lemma~\ref{Lemma.ClosureDistributive}. Since $\R$ does not add new subsets to $\kappa$, it follows that $\kappa$ remains measurable in $V[g][H]$.

But consider now $\HOD^{V[g][H]}$, which is the same as $\HOD^{V[T][H][b]}$. Since $\R$ forces to code $T$ explicitly into the \GCH\ or $\Diamond^*$ pattern, we know that $T\in\HOD^{V[g][H]}$. Meanwhile, since $T$ is weakly homogeneous in $V[T]$, it remains so in $V[T][H]$ because the automorphisms are still there, and consequently $\HOD^{V[g][H]}=\HOD^{V[T][H][b]}\of V[T][H]$ by Lemma~\ref{Lemma.WeaklyHomogenousControlsHOD}. Since $T$ is a $\kappa$-Suslin tree in $V[T]$ and hence in $V[T][H]$, as this model has the same subsets of $\kappa$ as $V[T]$, it follows that $\HOD^{V[g][H]}$ can have no branch through $T$, and so $\kappa$ does not have the tree property in $\HOD^{V[g][H]}$. Thus, $\kappa$ is not weakly compact in $\HOD^{V[g][H]}$, while it is measurable in $V[g][H]$, as desired.
\end{proof}

This argument generalizes to other large cardinals that can be made indestructible by $\Add(\kappa,1)$ (or even merely resurrectible after this forcing) and the coding forcing, allowing us to increase the gap in strength between the property exhibited by the large cardinal and the property it exhibits in $\HOD$. Let us illustrate in the case of a supercompact cardinal.

\begin{theorem}\label{Theorem.SupercompactNotWCinHOD}
 If $\kappa$ is a supercompact cardinal, then there is a forcing extension in which $\kappa$ remains supercompact, but is not weakly compact in  $\HOD$.
\end{theorem}

\begin{proof}
Suppose that $\kappa$ is a supercompact cardinal. By preparatory forcing, if necessary, we may assume without loss of generality that the supercompactness of $\kappa$ is indestructible by $\ltkappa$-directed forcing over $V$. We now force over $V$ to add a $V$-generic Cohen set $g\of\kappa$. As above, we may factor this forcing as $\Add(\kappa,1)\cong\SS*\Tdot$, first adding a weakly homogeneous $\kappa$-Suslin tree and then forcing with the tree. So the extension $V[g]$ can be viewed as $V[T][b]$, where $T$ is the generic $\kappa$-Suslin tree that is added and $b$ is the $V[T]$-generic branch through $T$. Let $\R$ be the forcing in $V[T]$ to code (a code for) $T$ into the \GCH\ or $\Diamond^*_\lambda$ patterns on the next $\kappa$ many regular cardinals above $\kappa$. Suppose that $H\of\R$ is $V[T]$-generic, which as before is equivalent to assuming that $H$ is $V[g]$-generic. We may view the extension $V[g][H]$ as $V[T][H][b]$. Since $\R$ is $\ltkappa$-directed closed in $V[T]$, it is also $\ltkappa$-directed closed in $V[g]$, since any subset of $\R$ of size less than $\kappa$ in $V[g]$ is in $V[T]$. Altogether, therefore, the forcing $\Add(\kappa,1)*\R$ is $\ltkappa$-directed closed in $V$, and so our indestructibility assumption on $\kappa$ ensures that $\kappa$ is supercompact in $V[g][H]$.

Meanwhile, we claim that $\kappa$ is not weakly compact in $\HOD^{V[g][H]}$. The tree $T$ is in $\HOD^{V[g][H]}$, since we explicitly forced to encode it. But since $V[g][H]=V[T][H][b]$ and $T$ is weakly homogeneous in $V[T]$ and hence in $V[T][H]$, it follows by Lemma~\ref{Lemma.WeaklyHomogenousControlsHOD} that $\HOD^{V[g][H]}\of V[T][H]$. Since the forcing to add $H$ adds no new subsets of $T$, it follows that $T$ has no cofinal branches in $V[T][H]$ and hence none in $\HOD^{V[g][H]}$, and so the tree property fails for $\kappa$ there, which means that $\kappa$ is not weakly compact in $\HOD^{V[g][H]}$, as desired.
\end{proof}

Let us now provide with Theorem~\ref{Theorem.GeneralVersionIndividualCardinal} a general template for the method, allowing us to extend the phenomenon to many other large cardinals. For this purpose, define that a property $\phi(\kappa)$ is {\df coding compatible}, if after forcing to add a Cohen subset $g\of\kappa$, then for any particular $E\of\kappa$ in $V[g]$ there is some further forcing $\R_E\in V[E]$, such that (i) forcing with $\R_E$ over $V[E]$ does not add subsets to $\kappa$; (ii) forcing with $\R_E$ over $V[g]$ makes $E$ ordinal definable; and (iii) forcing with $\R_E$ over $V[g]$ forces $\phi(\kappa)$. It will turn out that many large cardinal properties can be made coding compatible.

\begin{theorem}\label{Theorem.GeneralVersionIndividualCardinal}
 Suppose that the property $\phi(\kappa)$ is coding compatible. Then there is a forcing extension $V[g][H]$ in which $\phi(\kappa)$ holds, but $\kappa$ is not weakly compact in $\HOD^{V[g][H]}$.
\end{theorem}

\begin{proof}
Force with $\Add(\kappa,1)$ to add a $V$-generic Cohen subset $g\of\kappa$. By Lemma~\ref{Lemma.Kunen}, we may decompose $\Add(\kappa,1)$ as $\SS*\Tdot$ and view $V[g]=V[T][b]$, where $T$ is a weakly homogeneous $\kappa$-Suslin tree in $V[T]$ and $b\of T$ is a branch through it. In $V[T]$, let $E\of\kappa$ code $T$ in some absolute canonical manner. Since $\phi(\kappa)$ is coding compatible, there is a forcing notion $\R_E\in V[T]$, not adding subsets to $\kappa$ over $V[T]$, such that if $H\of\R_{E}$ is $V[g]$-generic, then $\phi(\kappa)$ holds in $V[g][H]$ and $E$ is ordinal definable there. So we have $T\in\HOD^{V[g][H]}$, since we made $E$ ordinal definable. Since $V[g][H]=V[T][H][b]$ and $T$ is weakly homogeneous in $V[T][H]$, it follows that $\HOD^{V[g][H]}\of V[T][H]$, and since $H$ does not add subsets to $\kappa$ over $V[T]$, it follows that $T$ still has no cofinal branches in $V[T][H]$, and hence none in $\HOD^{V[g][H]}$. Thus, $\kappa$ is not weakly compact in $\HOD^{V[g][H]}$.
\end{proof}

One could relax the template of Theorem~\ref{Theorem.GeneralVersionIndividualCardinal} to allow the additional forcing $\R_E$ after $\Add(\kappa,1)$ to add subsets to $\kappa$, provided that it was weakly homogeneous in $V[E]$. The point would be in this case that if $H\of\R_{E}$ is $V[g]$-generic, then we view $V[g][H]$ as $V[T][H][b]$, and the last two steps of forcing $H\times b\of\R_E\times T$ is weakly homogeneous in $V[T]$, which means $\HOD^{V[g][H]}\of V[T]$, where $T$ has no cofinal branches.

The utility of Theorem~\ref{Theorem.GeneralVersionIndividualCardinal} is revealed in Observation~\ref{Observation.LocalGlobal} and Corollary~\ref{Corollary.LCNotLargeInHOD}. For this purpose, define that a property $\phi(\kappa)$ is {\df locally verifiable}, if $\phi(\kappa)\Iff\exists\theta\, H_\theta\satisfies\psi(\kappa)$, for some assertion $\psi$. We say $\phi$ is {\df local}, if both $\phi$ and its negation $\neg\phi$ are locally verifiable. It is an excellent exercise to check that the locally verifiable properties are precisely the $\Sigma_2$-definable properties in set theory (see \cite{Hamkins2014:LocalPropertiesInSetTheory}), and consequently the local properties are the $\Delta_2$ properties, with large cardinal examples including measurability, superstrongness, almost hugeness and many others listed in Corollary~\ref{Corollary.LCNotLargeInHOD}.

\begin{observation}\label{Observation.LocalGlobal}\
 \begin{enumerate}
  \item Any locally verifiable property that holds in a forcing extension $V[g][h]$, where $g\of\kappa$ is a $V$-generic Cohen subset of $\kappa$ and $h$ is additional (possibly trivial) forcing not adding subsets to $\kappa$, is coding compatible.
  \item Any property that is indestructible by $\ltkappa$-directed closed forcing is coding compatible.
 \end{enumerate}
\end{observation}

\begin{proof} If a locally verifiable property $\phi(\kappa)$ holds in $V[g][h]$, where $g\of\kappa$ is a Cohen subset of $\kappa$ and $h\of\Q$ is further forcing not adding subsets to $\kappa$, then $\phi(\kappa)$ is witnessed in some $H_\theta^{V[g][h]}$. For any $E\of\kappa$ in $V[g]$, therefore, we may simply force over $V[g][h]$ to code $E$ into the \GCH\ or $\Diamond^*$ pattern above $\theta$, which will make $E$ definable while preserving $H_\theta$ and hence $\phi(\kappa)$, thereby showing that $\phi(\kappa)$ was coding compatible in $V$. And if $\phi(\kappa)$ is indestructible by 
$\ltkappa$-directed closed forcing, then it will be preserved by the forcing to add $g\of\kappa$ and then code $E$ above $\kappa$.
\end{proof}

One doesn't need full indestructibility, of course, but rather only that the property is preserved after adding the Cohen set and the coding forcing (plus, if necessary, the additional forcing not adding subsets to $\kappa$). For example, any measurable cardinal $\kappa$ can be made indestructible by $\Add(\kappa,1)$ followed by any further $\leqkappa$-distributive forcing. This perspective unifies the local case of Observation~\ref{Observation.LocalGlobal} with the indestructibility case.

\begin{corollary}\label{Corollary.LCNotLargeInHOD}
Suppose that $\kappa$ has any of the following large cardinal properties
\begin{quote}
\begin{itemize}
 \item[Local:] weakly compact, indescribable, totally indescribable, Ramsey, strongly Ramsey, measurable, $\theta$-tall, $\theta$-strong, Woodin, $\theta$-supercompact, 
superstrong, $n$-superstrong, $\omega$-superstrong, $\lambda$-extendible, almost huge, huge,
$n$-huge, rank-into-rank ($I_0, I_1$ and $I_3$);
 \item[Global:] unfoldable, strongly unfoldable, tall, strong, supercompact, superhuge, and many others.
\end{itemize}
\end{quote}
Then there is a forcing extension in which $\kappa$ continues to have the property, but is not weakly compact in \HOD.
\end{corollary}

\begin{proof} The point is that each of these large cardinal properties can be made coding compatible. In the case of the local properties, standard arguments in the literature show how these large cardinals are preserved by suitable Easton-support iterations of Cohen forcing, forcing to add a Cohen subset, for example, at every regular cardinal stage; this is essentially the canonical forcing of the \GCH. (For the superstrong, extendible and rank-into-rank cases, the reader may find it helpful to consult \cite{DimonteFriedman2014:RankIntoRankHypothesesAndTheFailureOfGCH, Friedman2006:LargeCardinalsAndLLikeUniverses, Tsaprounis2013:OnExtendibleCardinalsAndGCH, Hamkins2001:WholenessAxiom}.) If $\kappa$ has any of these local properties, and $V[G]$ is the forcing extension up to stage $\kappa$, then after forcing to add $g\of\kappa$ at stage $\kappa$ and possibly additional forcing $g^*\of\Q^*$---in the case of superstrong, almost huge and others one must force with the rest of the iteration above $\kappa$---then $\kappa$ retains the large cardinal property in $V[G][g][g^*]$, and so by Observation~\ref{Observation.LocalGlobal}, the property was made coding compatible in $V[G]$. Consequently, these instances of the corollary follow from Theorem~\ref{Theorem.GeneralVersionIndividualCardinal}.

For the global properties, as mentioned before the statement of the corollary, one needs to have the large cardinal property after forcing with $\Add(\kappa,1)$ and the coding forcing for a specific $E\of\kappa$ (plus, if this helps, additional forcing not adding subsets to $\kappa$). Results in the literature establish the required degree of indestructibility(see variously \cite{Laver78, GitikShelah89, Hamkins2000:LotteryPreparation, HamkinsJohnstone2010:IndestructibleStrongUnfoldability, Hamkins2009:TallCardinals,GitmanJohnstone:IndestructibilityForRamsey, Cody2013:EastonsTheoremInThePresenceOfWoodinCardinals, Hamkins:ForcingAndLargeCardinals}).
\end{proof}

Although our method is flexible, it does not apply to all large cardinals, and there are several cases left open. For example, can there be an extendible cardinal, which is not extendible in \HOD? (Or not measurable in \HOD? Not weakly compact in \HOD?) The method also does not apply to those large cardinals, such as the strongly compact cardinals, which lack a robust forcing-preservation theory (although of course the case of supercompactness in Corollary~\ref{Corollary.LCNotLargeInHOD} provides a strongly compact cardinal that is not weakly compact in \HOD). See Section~\ref{Section.Limitations} for these and other open questions.

\section{Proper class of large cardinals, not large in \HOD}\label{Section.ProperClassLC}

We should like now to extend the phenomenon to have a proper class of various kinds of large cardinals, which are not large in $\HOD$, and furthermore such that there are no such large cardinals in $\HOD$. Let us begin with the case of measurable cardinals.

\begin{theorem}\label{Theorem.ProperClassMeasurablesNotLargeInHOD}\
There is a class forcing notion $\P$ forcing that
 \begin{enumerate}
  \item All measurable cardinals of the ground model are preserved and no new measurable cardinals are created.
  \item There are no measurable cardinals in the $\HOD$ of the extension.
  \item The measurable cardinals of the extension are not weakly compact in the $\HOD$ of the extension.
 \end{enumerate}
One may also ensure that the \GCH\ holds in the extension and its $\HOD$.
\end{theorem}

\begin{proof}
In the interesting case, there are many measurable cardinals in the ground model $V$. Let $\Vbar=V[F]$ be the extension arising from the forcing mentioned in the proof of Lemma~\ref{Lemma.MeasurableCCAcoding}, so that the \GCH\ holds at every inaccessible cardinal in $\Vbar$ and every set of ordinals of $\Vbar$ is coded into the \GCH\ or $\Diamond^*$ patterns at, say, the triple successors $\delta^{+++}$ of the $\beth$-fixed points $\delta=\beth_\delta$, and furthermore that every measurable cardinal $\kappa$ is indestructible by forcing over $\Vbar$ with $\Add(\kappa,1)$. Now, in $\Vbar$, for each measurable cardinal $\kappa$, consider the forcing $\Add(\kappa,1)*\Rdot(\kappa)$, where as previously we factor $\Add(\kappa,1)$ into two steps $\SS_\kappa*\Tdot_\kappa$, which first adds a homogeneous $\kappa$-Suslin tree $T_\kappa$ and then forces with it, and then $\Rdot(\kappa)$ is the forcing in $\overline{V}[T_\kappa]$ that codes the tree $T_\kappa$ into the \GCH\ or $\Diamond^*$ patterns at the next $\kappa$ many regular cardinals above $\kappa$, starting above $\kappa^{+++}$. Thus, $\Add(\kappa,1)*\Rdot(\kappa)\iso\SS_\kappa*(\Rdot(\kappa)\dot\times\Tdot_\kappa)$. Let $\P=\Pi_\kappa(\Add(\kappa,1)*\Rdot(\kappa))$ be the Easton-support product of these forcing notions, taken over all measurable cardinals $\kappa$, and let $G\of\P$ be $\Vbar$-generic. Our final model is $\Vbar[G]$, which we shall now argue is as desired.

First, we claim that every measurable cardinal $\kappa$ is preserved to $\Vbar[G]$. Since the forcing above $\kappa$ is $\leqkappa$-closed, and the forcing $\R(\kappa)$ is $\leqkappa$-distributive after adding the Cohen subset to $\kappa$, it suffices to argue that $\kappa$ is measurable in the extension $\Vbar[G_\kappa][g_\kappa]$, where $G_\kappa$ performs the forcing at measurable cardinals below $\kappa$ and $g_\kappa\of\kappa$ is the Cohen subset of $\kappa$ added by $\Add(\kappa,1)$ on coordinate $\kappa$. By our indestructibility assumption on $\kappa$ in $\Vbar$, we know that $\kappa$ remains measurable in $V_1=\Vbar[g_\kappa]$, and so it suffices to argue merely that the forcing $\P_\kappa$ preserves the measurability of $\kappa$, forcing over $V_1$. And this can be done in the style of the Kunen-Paris theorem. Namely, fix in $V_1$ any normal ultrapower embedding $j:V_1\to M$ for which $\kappa$ is not measurable in $M$, and consider $j(\P_\kappa)=\P_\kappa\times\P_{\kappa,j(\kappa)}$, where $\P_{\kappa,j(\kappa)}$ is the rest of the product forcing from stage $\kappa$ up to $j(\kappa)$ in $M$. Since $\kappa$ is not measurable in $M$, there is no forcing at coordinate $\kappa$, and so $\P_{\kappa,j(\kappa)}$ is $\leqkappa$-closed in $M$. Since $M^\kappa\of M$ in $V_1$ and $\P_{\kappa,j(\kappa)}$ has $|j(2^\kappa)|^{V_1}=\kappa^+$ many dense subsets in $M$, we may construct by diagonalization in $V_1$ an $M$-generic filter $G_{\kappa,j(\kappa)}\of\P_{\kappa,j(\kappa)}$. It follows that $G_\kappa\times G_{\kappa,j(\kappa)}\of j(\P_\kappa)$ is $M$-generic, and so we may lift the embedding to $j:V_1[G_\kappa]\to M[j(G_\kappa)]$, where $j(G_\kappa)=G_\kappa\times G_{\kappa,j(\kappa)}$, thereby witnessing that $\kappa$ is measurable in $V_1[G_\kappa]$ and hence also in $\Vbar[G]$, as desired.

Next, since $\Vbar=V[F]$, where we had forced over the original ground model $V$ with $F\of\bar\P$ of Lemma~\ref{Lemma.MeasurableCCAcoding}, then the combined forcing $\bar\P*\P$ admits a closure point in the sense of~\cite{Hamkins2003:ExtensionsWithApproximationAndCoverProperties}, and so by the main theorem of \cite{Hamkins2003:ExtensionsWithApproximationAndCoverProperties} therefore creates no new measurable cardinals. In particular, the three models $V\of \Vbar\of \Vbar[G]$ have the same measurable cardinals. The same reasoning shows that the large cardinals in $\Vbar[\vec T]$, where $\vec T$ is the sequence of Suslin trees $T_\kappa$ added by $\SS_\kappa$ at coordinate $\kappa$, are also large in $\Vbar$.

We claim that $\HOD^{\Vbar[G]}$ has no measurable cardinals. To see this, we argue that $\HOD^{\Vbar[G]}=\Vbar[\vec T]$. The forward inclusion is a consequence of the fact that $\Pi_\kappa(\R(\kappa)\times T_\kappa)$ is weakly homogeneous in $\Vbar[\vec T]$, since the coding forcing is weakly homogeneous and the trees themselves are weakly homogeneous, and so $\HOD^{\Vbar[G]}\of \Vbar[\vec T]$ by Lemma~\ref{Lemma.WeaklyHomogenousControlsHOD}. Conversely, note that the coding performed by $F$ at the triple successors of the $\beth$-fixed points is preserved to $\Vbar[G]$, since the forcing to add $G$ does not interfere with that coding and furthermore all $\beth$-fixed points are preserved from $\Vbar$ to $\Vbar[G]$. It follows that every set of ordinals in $\Vbar$ is coded into the \GCH\ or $\Diamond^*$ pattern on such cardinals in $\Vbar[G]$, and this implies $\Vbar\of\HOD^{\Vbar[G]}$. Further, the trees $T_\kappa$ themselves are coded into the \GCH\ or $\Diamond^*$ pattern on the next $\kappa$-many regular cardinals of $\Vbar[G]$, and so $\vec T$ is definable in $\Vbar[G]$. Thus, $\Vbar[\vec T]\of \HOD^{\Vbar[G]}$, and we conclude $\HOD^{\Vbar[G]}=\Vbar[\vec T]$.

So we shall show that there are no measurable cardinals in $\Vbar[\vec T]$. Suppose that $\kappa$ is a measurable cardinal in $\Vbar[\vec T]$. By our remark two paragraphs above, it follows that $\kappa$ was measurable in $\Vbar$ and therefore is one of the coordinates at which forcing is performed. In particular, at stage $\kappa$ we added the $\kappa$-Suslin tree $T_\kappa$, which is Suslin in $\Vbar[T_\kappa]$. The forcing above $\kappa$ cannot affect whether $T_\kappa$ is $\kappa$-Suslin, since it adds no new subsets to $\kappa$. The forcing $\Pi_{\delta<\kappa}\SS_\delta$ that adds the trees $T_\delta$ at measurable cardinals $\delta<\kappa$ is productively $\kappa$-c.c.~(meaning it remains $\kappa$-c.c.~in the forcing extension), and such forcing cannot add a $\kappa$-branch through a $\kappa$-Suslin tree. Thus, the tree property fails for $\kappa$ in $\Vbar[\vec T]$ and in particular, $\kappa$ is not weakly compact there. So there are no measurable cardinals in $\Vbar[\vec T]$, and consequently no measurable cardinals in $\HOD^{\Vbar[G]}$, establishing statement~(2). Furthermore, the measurable cardinals of $\Vbar[G]$ are the same as the measurable cardinals of $\Vbar$, which are not weakly compact in $V[\vec T]=\HOD^{\Vbar[G]}$, establishing statement~(3).

Finally, to achieve the \GCH\ in the extension $\Vbar[G]$ and its \HOD, one should start in a model of \GCH\ and use exclusively the $\Diamond^*$ coding, rather than the \GCH\ coding.
\end{proof}

\begin{theorem}\label{Theorem.ProperClassSCNotLargeInHOD}\
There is a class forcing notion $\P$ forcing that
 \begin{enumerate}
  \item All supercompact cardinals of the ground model are preserved and no new supercompact cardinals are created.
  \item There are no supercompact cardinals in the $\HOD$ of the extension.
  \item The supercompact cardinals of the extension are not weakly compact in the $\HOD$ of the extension.
 \end{enumerate}
One may also ensure that the \GCH\ holds in the extension and its $\HOD$.
\end{theorem}

\begin{proof}
We simply modify the proof of Theorem~\ref{Theorem.ProperClassMeasurablesNotLargeInHOD} for the supercompact context. In the interesting case, the ground model $V$ has many supercompact cardinals. We may assume without loss that they have all been made Laver indestructible, by forcing with the global Laver preparation (or if the \GCH\ is desired, we may assume that they are all indestructible by \GCH-preserving directed closed forcing, as in \cite{Hamkins98:AsYouLikeIt}). When there is a proper class of supercompact cardinals, this implies outright by Observation~\ref{Observation.IndestructibleImpliesV_kappaSubsetHOD} that every set of ordinals in $V$ is coded explicitly into the \GCH\ pattern (and the same idea applies also to the $\Diamond^*$ pattern when \GCH\ is desired), at whichever type of coding cardinals we might prefer. If the supercompact cardinals are bounded, then we may simply perform additional coding forcing above the bound to ensure that every set is coded in this way as desired. Let $\P$ be the Easton-support product $\Pi_\kappa(\Add(\kappa,1)*\Rdot(\kappa))$ as in Theorem~\ref{Theorem.ProperClassMeasurablesNotLargeInHOD}, but now with forcing only at coordinates $\kappa$ that are fully supercompact. The indestructibility of $\kappa$ in $V$ ensures that the forcing $\P^\kappa$ at coordinates $\kappa$ and above preserves the supercompactness of $\kappa$ to $V[G^\kappa]$, and so we must prove only that $\P_\kappa$ preserves the supercompactness of $\kappa$. If $\kappa$ is not a limit of supercompact cardinals, then this product is small relative to $\kappa$ and therefore preserves the supercompactness of $\kappa$. For the remaining case, assume that $\kappa$ is a supercompact limit of supercompact cardinals. Since the \SCH\ holds above any supercompact cardinal, we may fix a large strong limit cardinal $\theta$ with $2^\theta=\theta^+$ and with cofinality above $\kappa$, so that $\theta^{\ltkappa}=\theta$, and let $j:V[G^\kappa]\to M$ be a $\theta$-supercompactness embedding in $V[G^\kappa]$, for which $\kappa$ is not $\theta$-supercompact in $M$. It follows that there are no supercompact cardinals in $M$ in the interval $[\kappa,\theta]$, and so $j(\P_\kappa)$ has no forcing at coordinates in the interval $[\kappa,\theta]$. So $j(\P_\kappa)$ factors as $\P_\kappa\times\P_{\kappa,j(\kappa)}$, but the second factor is $\leqtheta$-closed in $M$ and has size $j(\kappa)$. Since $M^\theta\of M$ in $V[G^\kappa]$ and furthermore $|j(\kappa)|^{V[G^\kappa]}=\theta^+$, it follows that in $V[G^\kappa]$ we may construct an $M$-generic filter $G^*\of\P_{\kappa,j(\kappa)}$, and so when this filter is combined with $G_\kappa\of\P_\kappa$, we may lift the embedding to $j:V[G^\kappa][G_\kappa]\to M[j(G_\kappa)]$, where $j(G_\kappa)=G_\kappa\times G^*$. This lifted embedding witnesses that $\kappa$ remains $\theta$-supercompact in $V[G^\kappa][G_\kappa]=V[G]$, as desired. Meanwhile, we may argue as in the proof of Theorem~\ref{Theorem.ProperClassMeasurablesNotLargeInHOD} that $\HOD^{V[G]}=V[\vec T]$, which is a model in which no supercompact cardinal $\kappa$ of $V$ and hence of $V[G]$ is weakly compact. And further, no new supercompact cardinals are created in $V[\vec T]$, and so $\HOD^{V[G]}$ has no supercompact cardinals at all.
\end{proof}

An essentially similar argument works with strong cardinals; we omit the proof.

\begin{theorem}\label{Theorem.ProperClassStrongNotLargeInHOD}\
There is a class forcing notion $\P$ forcing that
 \begin{enumerate}
  \item All strong cardinals of the ground model are preserved and no new strong cardinals are created.
  \item There are no strong cardinals in the $\HOD$ of the extension.
  \item The strong cardinals of the extension are not weakly compact in the $\HOD$ of the extension.
 \end{enumerate}
One may also ensure that the \GCH\ holds in the extension and its $\HOD$.
\end{theorem}

Indeed, the method works with many other kinds of large cardinals, including strongly unfoldable cardinals, strongly Ramsey cardinals, tall cardinals and many others, which we invite the reader to check. Let us consider the case of weakly compact cardinals.

\begin{theorem}\
There is a class forcing notion $\P$, preserving all weakly compact cardinals, creating no new weakly compact cardinals, and forcing that there are no weakly compact cardinals in the \HOD\ of the extension. One may also ensure that the \GCH\ holds in the extension and its \HOD.
\end{theorem}

\begin{proof}
Once again, in the interesting case there will be many weakly compact cardinals in $V$. As in Theorem~\ref{Theorem.ProperClassMeasurablesNotLargeInHOD}, we force to $\Vbar=V[F]$ in such a way that preserves all weakly compact cardinals and creates no new weakly compact cardinals, such that every set of ordinals in $\Vbar$ is coded into the \GCH\ or $\Diamond^*$ patterns at the triple successors of the $\beth$-fixed points, and such that every weakly compact cardinal $\kappa$ of $\Vbar$ is indestructible by further forcing with $\Add(\kappa,1)$. Next, let $\P$ be the Easton-support product $\Pi_{\kappa}(\Add(\kappa,1)*\Rdot(\kappa))$, taken over all weakly compact cardinals $\kappa$, where $\Rdot(\kappa)$ is as before the forcing to code the $\kappa$-Suslin tree added by the first factor into the \GCH\ or $\Diamond^*$ patterns at the next $\kappa$ many regular cardinals above $\kappa$, starting above $\kappa^{+++}$ (one could alternatively perform forcing at every inaccessible cardinal). Suppose that $G\of\P$ is $\Vbar$-generic and consider the extension $\Vbar[G]$, our final model. One may see that every weakly compact cardinal of $V$ is preserved to $\Vbar[G]$, using the same argument as in the case of measurable cardinals in Theorem~\ref{Theorem.ProperClassMeasurablesNotLargeInHOD}, namely, if $\kappa$ is weakly compact in $\Vbar$, then it suffices to argue that $\kappa$ is weakly compact in $\Vbar[G_\kappa][g_\kappa]$. The forcing to add the Cohen set $g_\kappa$ preserves weak compactness by the indestructibility assumption on $\Vbar$, and one may lift any weak compactness embedding $j:M\to N$ of $\Vbar[g_\kappa]$ to $j:M[G_\kappa]\to N[j(G_\kappa)]$ by diagonalizing to produce $G_{\kappa,j(\kappa)}$ as in the measurable case (this works whether or not there is forcing at stage $\kappa$ in $j(\P_\kappa)$, and one does not need \GCH\ at $\kappa$ here in the weakly compact case). Thus, all weakly compact cardinals of $V$ are preserved to $\Vbar[G]$. No new weakly compact cardinals are created in $\Vbar[G]$ (or in $\Vbar[\vec T]$) by the main result of~\cite{Hamkins2003:ExtensionsWithApproximationAndCoverProperties}. Finally, $\HOD^{\Vbar[G]}=\Vbar[\vec T]$ for the same reasons as in the previous cases, and this is a model having no weakly compact cardinals, for the reasons explained in the earlier cases.
\end{proof}

\section{Questions and Limitations}\label{Section.Limitations}

There are numerous natural questions left open by the results of this paper. Several large cardinal notions, such as the extendible cardinals, are missing from Corollary~\ref{Corollary.LCNotLargeInHOD}. Furthermore, although in Section~\ref{Section.ProperClassLC} we arranged models with a proper class of large cardinals of a particular type, with none of that type in \HOD, what we did not arrange is that \HOD\ had no large cardinals of smaller type.

For example, those results did not provide a model with a measurable cardinal, but no weakly compact cardinal in \HOD. Such a situation, however, is impossible. Specifically, Philip Welch pointed out that if kappa is an $\omega$-\Erdos\ cardinal in $V$, then it is $\omega$-\Erdos\ in any inner model (by an absoluteness argument), and therefore there are also weakly compact cardinals below $\kappa$ in any inner model. Similarly, if $\alpha$ is countable in $\HOD$ and there is an $\alpha$-\Erdos\ cardinal in $V$, then there is one in $\HOD$ by the same absoluteness argument. A similar observation about $\alpha$-\Erdos\ cardinals was made in \cite[p.~3]{Friedman2002:OsharpAndInnerModels}. In the same vein, Gitik and Hamkins \cite{GitikHamkins:LargeCardinalsNecessarilyLargeInHOD} proved that if $\kappa$ is a measurable cardinal, then for any transitive inner model $W$, including $W=\HOD$, the set of cardinals below $\kappa$ that are weakly compact, ineffable, and even superstrongly $\ltkappa$-unfoldable has normal measure one in $V$; indeed, if $\kappa$ is merely subtle, then the set of cardinals below $\kappa$ that are superstrongly $\ltkappa$-unfoldable in $W$ is stationary in $V$ and hence also in $W$.

Meanwhile, many other instances remain open.

\begin{question}\label{Question.LCbutNoneInHOD?}
 Can there be a strong cardinal, or a proper class of strong cardinals, but no measurable cardinal in \HOD? Can there be an extendible cardinal which is not weakly compact in \HOD? Can there be a proper class of extendible cardinals, but no supercompact cardinal in $\HOD$? Can there be a supercompact cardinal, which is strongly compact but not supercompact in $\HOD$?
\end{question}

There are infinitely many variations of these questions. In the supercompact cardinal context, W. Hugh Woodin has announced in email correspondence with the first and third authors the following remarkable implication, for which he has sketched a proposed proof making use of the \HOD\ dichotomy theorem of \cite{Woodin2010:SuitableExtenderModelsI}.

\begin{conjecture}[Woodin]\label{Conjecture.Woodin}
 If there is a supercompact cardinal, then there is a measurable cardinal in \HOD.
\end{conjecture}

This result would provide a general limitation on the phenomenon we have been exploring in this paper, and one may view Question~\ref{Question.LCbutNoneInHOD?} as asking whether or not there are analogues of Woodin's conjecture at large cardinals below supercompact.

The theme of this paper has been to show that there can be cardinals that are large in $V$, but small in $\HOD$. A dual theme would be to ask for cardinals that are large in $\HOD$, but small in $V$. On this note, Cummings, Friedman and Golshani \cite[thm 1.1]{CummingsFriedmanGolshani:CollapsingTheCardinalOfHOD} proved that $\alpha^+$ can be greater than $(\alpha^+)^\HOD$ for all cardinals $\alpha$ below a measurable cardinal, and this can happen for a club of
cardinals $\alpha$ below a supercompact cardinal. They have conjectured that $\alpha^+$ can be measurable in $\HOD$ for club-many cardinals $\alpha$  below a supercompact cardinal. In email correspondence, Woodin has conjectured that ``club-many" cannot be replaced by ``all", even with ``measurable" replaced by ``inaccessible".

It is also natural to broaden the theme of this paper beyond $\HOD$ and definability altogether, by asking whether the results extend to other natural inner models and extensions. Consider, for example, the stable core $\mathbb S$ of \cite{Friedman2012:TheStableCore}. As $V$ is generic over $\mathbb S$, one might expect results for the stable core similar to those we have obtained for $\HOD$. This would present new challenges, however, as it is far more difficult to code a set into the stable core than it is to make it ordinal-definable. Meanwhile, we also consider the questions for ground models or inner models generally:

\begin{question}\label{Question.VsubsetW}
Can there be a supercompact cardinal in a forcing extension $V[G]$, if there are no measurable cardinals in $V$?
\end{question}

There are again innumerable variations of these questions to other large cardinal concepts, and we take ourselves to have asked an entire scheme of questions here. For the strong-cardinal analogue of Question~\ref{Question.VsubsetW}, we have the sketch of a proof that there can be a model with no measurable cardinals, but with a strong cardinal in a forcing extension.

The second author proved in \cite{Friedman2002:OsharpAndInnerModels} that if there are no inaccessible cardinals in $L[0^\sharp]$, then there is an $M \subsetneq L[0^\sharp]$ in which \GCH\ holds and there is no cardinal $\kappa$ which is $\kappa$-Mahlo. It is natural to ask whether similar results hold for larger core models.

\section{Appendix: background material}\label{Section.Background}

We have made use in this article of several general facts and some other background material, whose proofs we include here for completeness.

In the main argument, we used the following fact, due to Kunen, that the forcing $\Add(\kappa,1)$ to add a Cohen subset to an inaccessible cardinal $\kappa$ can be factored as first adding a certain weakly homogeneous $\kappa$-Suslin tree and then forcing with that tree.

\begin{lemma}[Kunen \cite{Kunen78:SaturatedIdeals}]\label{Lemma.Kunen}
If $\kappa$ is inaccessible, then there is a strategically $\ltkappa$-closed notion of forcing $\SS$ of size $\kappa$ such that forcing with $\SS$ adds a weakly homogeneous $\kappa$-Suslin tree $T$ and the combined forcing $\SS*\Tdot$ is forcing-equivalent to the forcing $Add(\kappa,1)$ to add a Cohen subset of $\kappa$.
\end{lemma}

\noindent See also the detailed accounts in \cite{CodyGitikHamkinsSchanker2015:LeastWeaklyCompact,GitmanWelch2011:Ramsey-likeCardinalsII}. The key point is that $\SS*\Tdot$ has a $\ltkappa$-closed dense subset of size $\kappa$, and all such (nontrivial) forcing is equivalent to $\Add(\kappa,1)$.

A forcing notion $\Q$ is {\df weakly homogeneous}, if for any two conditions $p,q\in\Q$, there is an automorphism $\pi$ of $\Q$ for which $\pi(q)$ and $p$ are compatible. It follows that if $\varphi$ is a statement in the forcing language involving only check names $\check x$, and some condition $p$ forces $\varphi$, then every condition forces $\varphi$, since otherwise some $q$ forces $\neg\varphi$, but in this case $\pi(q)$ will also force $\neg\varphi$, which is impossible if it is compatible with $p$. A weaker variant of this property is that $\Q$ is {\df locally homogeneous}, if for any $p,q\in\Q$, there are extensions $p^*\leq p$ and $q^*\leq q$ and an isomorphism $\pi$ of $\Q\restrict p^*$ with $\Q\restrict q^*$; it follows again in this case that all conditions force the same assertions with only ground model parameters.

\begin{lemma}[Folklore]\label{Lemma.WeaklyHomogenousControlsHOD}
 If\/ $\Q$ is a locally homogeneous notion of forcing and $G\of\Q$ is $V$-generic, then $\HOD^{V[G]}\of\HOD(\Q)^V$. In particular, if $\Q$ is also ordinal definable in $V$, then $\HOD^{V[G]}\of\HOD^V$.
\end{lemma}

The point is that if $A\of\Ord$ is defined in $V[G]$ by $\alpha\in A\iff \varphi(\alpha,\beta)$, then we can define $A$ in the ground model as $\set{\alpha\st \one\forces\varphi(\check\alpha,\check\beta)}$, since all conditions must force the same assertions.

Several of our arguments use Easton's lemma:

\begin{lemma}[{\cite[lemma 15.17]{Jech:SetTheory3rdEdition}}]\label{Lemma.ClosureDistributive}
 Suppose that $G\times H$ is $V$-generic for $\P\times\Q$, where $\P$ is $\ltkappa$-closed and $\Q$ is $\kappa$-c.c.  Then $\P$ is $\ltkappa$-distributive in $V[H]$. In other words, $\Ord^\ltkappa\intersect V[G][H]\of V[H]$.
\end{lemma}

In our main argument, we often have need to perform forcing that ensures that a given set of ordinals becomes ordinal definable in the forcing extension. The general idea goes back to McAloon \cite{McAloon1971:ConsistencyResultsAboutOrdinalDefinability}, who forced $V=\HOD$ by various types of coding forcing. One quite commonly sees the forcing to code a given set into the \GCH\ pattern on a definable block of cardinals. This \GCH-controlling forcing is highly homogeneous and directed closed, and it is easy to see that the coding is preserved by any forcing below the place where the coding took place, since small forcing cannot affect the continuum function up high. But there are other coding methods, such as coding into the $\Diamond^*_\kappa$ pattern. The second author and his student Andrew Brook-Taylor \cite{Brooke-Taylor:Thesis,Brooke-Taylor2009:LargeCardinalsAndDefinableWellOrders} have emphasized the value of coding via the $\Diamond^*_\kappa$ pattern, in part because this can be done while preserving the \GCH, as well as various large cardinals. The main fact is the following. Recall that $\<D_\alpha\mid\alpha<\kappa>$ is a {\df $\Diamond^*_\kappa$-sequence}, where $\kappa$ is an uncountable regular cardinal, if $D_\alpha$ consists of at most $|\alpha|$ many subsets of $\alpha$ and for every $A\of\kappa$ there is a club $C\of\kappa$ such that $A\intersect\alpha\in D_\alpha$ for each $\alpha\in C$.

\begin{lemma}\label{Lemma.Diamond*Coding} One may force so as to control $\Diamond^*_\kappa$ as desired on the successor cardinals.
\begin{enumerate}
 \item \cite{Devlin1979:VariationsOnDiamond,Brooke-Taylor2009:LargeCardinalsAndDefinableWellOrders} If $\kappa$ is any regular cardinal, then $\Add(\kappa,\kappa^+)$ forces $\neg\Diamond^*_\kappa$.
 \item \cite{CummingsForemanMagidor2001:SquaresScalesAndStationaryReflection,Brooke-Taylor2009:LargeCardinalsAndDefinableWellOrders}
     For any infinite successor cardinal $\kappa^+$, there is a $\ltkappa^+$-closed forcing notion of size $2^{\kappa^+}$, which forces $\Diamond^*_{\kappa^+}$.
\end{enumerate}
\end{lemma}

Furthermore, the $\Diamond^*$ coding forcing in each case is locally homogeneous. The application of $\Diamond^*$ coding in this article requires that the coding is preserved by small forcing, so let us also note that here.

\begin{lemma}\label{Lemma.Diamond*InvariantBySmallForcing}
For any regular cardinal $\kappa$, the $\Diamond^*_\kappa$ hypothesis is invariant by small forcing. Furthermore, $\Diamond^*_\kappa$ is preserved by any forcing of size at most $\kappa$ preserving the regularity of $\kappa$.
\end{lemma}

\begin{proof} More specifically, what we claim is the following:
\begin{enumerate}
  \item For any small forcing extension $V[G]$, meaning that $G\of\P$ is $V$-generic for some $\P$ of size less than $\kappa$ in $V$, the principle  $\Diamond^*_\kappa$ holds in $V$ if and only if it holds in $V[G]$.
  \item If $\Diamond^*_\kappa$ holds in $V$ and $G\of\P$ is $V$-generic, where $\P$ has size at most $\kappa$ and $\kappa$ is regular in $V[G]$, then $\Diamond^*_\kappa$ continues to hold in $V[G]$.
\end{enumerate}
For (1), suppose that $\<D_\alpha \mid \alpha < \kappa>$ is a $\Diamond^*_\kappa$-sequence in $V[G]$, where $G\of\P$ is $V$-generic and $|\P|<\kappa$. Fix a name $\dot D$ for the sequence, forced by $\one$ via $\P$ to be a $\Diamond_\kappa^*$-sequence, and corresponding names $\dot D_\alpha$ for each $D_\alpha$. In $V$, for any infinite ordinal $\alpha$ above $|\P|$, let $B_\alpha$ consist simply of all those subsets $X\of\alpha$ for which there is some condition $p\in\P$ forcing $\check X\in\dot D_\alpha$. Since $\P$ is small, it follows that $B_\alpha$ has size at most $|\alpha|$, since any given condition can force at most $|\alpha|$ many sets into $\dot D_\alpha$. If $X\of\kappa$ is any subset of $\kappa$ in $V$, then there is a club $C\of\kappa$ in $V[G]$ such that $X\intersect\alpha\in D_\alpha$ for all $\alpha\in C$. Since these facts are forced by some condition, it follows that $X\intersect\alpha\in B_\alpha$ as well, and since the forcing is small, the club $C$ must contain a club in the ground model $V$, and so $\<B_\alpha\mid\alpha<\kappa>$ witnesses $\Diamond^*_\kappa$ in $V$, establishing the converse implication of statement (1).

The forward implication of (1) is generalized by statement (2), which we now prove. Suppose that $\<B_\alpha \mid \alpha<\kappa>$ is a $\Diamond^*_\kappa$-sequence in $V$ and $G\of\P$ is $V$-generic for forcing $\P$ of size at most $\kappa$ that preserves the regularity of $\kappa$. We may assume that $\P\of\kappa$. In the forcing extension $V[G]$, let $D_\alpha$ be the set of all subsets $X\of \alpha$ in $V[G]$ for which there is a nice $\P\intersect\alpha$-name $\dot X$ such that $\dot X_G=X$ and $\dot X$ is coded by an element of $B_\alpha$ (we use nice names simply to make the coding more transparent). If $X\of\kappa$ is in $V[G]$, then it has a nice $\P$-name $\dot X$, which is coded by a subset $X_0\of\kappa$ in $V$, and $X_0\intersect\alpha\in D_\alpha$ on a club of $\alpha<\kappa$. Since also $X_0\intersect\alpha$ codes a nice $\P\intersect\alpha$-name $\dot X_\alpha$ for subset of $\alpha$, on a club of $\alpha$, and furthermore $(\dot X_\alpha)_G=X\intersect\alpha$ for a club of $\alpha$, we have a club of $\alpha$ in $V[G]$ for which $X\intersect\alpha\in D_\alpha$, thereby witnessing $\Diamond^*_\kappa$ in $V[G]$.
\end{proof}

The following lemma is used in the proof of Theorem~\ref{Theorem.ProperClassMeasurablesNotLargeInHOD}.

\begin{lemma}\label{Lemma.MeasurableCCAcoding}
There is a class forcing notion $\P$ such that if $F\of\P$ is $V$-generic, then
 \begin{enumerate}
  \item $V$ and $V[F]$ have exactly the same measurable cardinals.
  \item The \GCH\ holds in $V[F]$ at every inaccessible cardinal.
  \item Every set of ordinals in $V[F]$ arises as the \GCH\ or $\Diamond^*$ pattern (as desired) on a block of cardinals of the form $\delta^{+++}$ for $\delta=\beth_\delta$.
  \item Every measurable cardinal $\kappa$ in $V[F]$ is indestructible by the forcing $\Add(\kappa,1)$.
 \end{enumerate}
By using $\Diamond^*$ coding and no \GCH\ coding, one can ensure \GCH\ in $V[F]$.
\end{lemma}

The proof of the lemma is extremely flexible, and one can use different coding points.

\begin{proof}Let $\P$ be the Easton-support proper-class forcing iteration that forces the \GCH\ at all infinite cardinals (including adding a Cohen subset at inaccessible cardinal stages), except that at stage $\eta=\delta^{+++}$ where $\delta=\beth_\delta$,
the forcing uses the lottery sum of the forcing $\Add(\eta^+,1)$, which forces the \GCH\ at $\eta$, with the forcing $\Add(\eta,\eta^{++})$, which forces a violation of the \GCH\ at $\eta$. Suppose that $F\of\P$ is $V$-generic, and consider the extension $V[F]$. The standard arguments show that every measurable cardinal $\kappa$ of $V$ is preserved to $V[F]$, and furthermore becomes indestructible by $\Add(\kappa,1)$ (and this is why we add a Cohen set at inaccessible stages). The main results of \cite{Hamkins2003:ExtensionsWithApproximationAndCoverProperties} show that the forcing creates no new measurable cardinals, since the iteration admits a closure point below the first inaccessible cardinal. The forcing ensures the \GCH\ at all cardinals except those coding points, and a simple density argument shows that every set of ordinals in the extension is coded into the \GCH\ pattern at those coding points. So the extension exhibits all the desired features. By using $\Diamond^*$ coding instead of \GCH\ coding, an essentially similar argument also achieves \GCH\ in $V[F]$.
\end{proof}

A cardinal $\kappa$ is {\df $(\Sigma_2,0)$-extendible} if $V_\kappa\elesub_{\Sigma_2} V_\theta$ for some ordinal $\theta$ (see~\cite{BagariaHamkinsTsaprounisUsuba:SuperstrongAndOtherLargeCardinalsAreNeverLaverIndestructible}).

\begin{observation}[{\cite[Theorem~4]{ApterFriedman:HODSupercompactnessIndestructibilityLevelByLevelEquivalence}}]\label{Observation.IndestructibleImpliesV_kappaSubsetHOD}
 If $\kappa$ is a Laver indestructible supercompact cardinal, then $V_\kappa\of\HOD$. In particular, in this case there are many measurable and partially supercompact cardinals in \HOD. More specifically, if $\kappa$ is Laver indestructibly $\Sigma_2$-reflecting, or even merely indestructibly $(\Sigma_2,0)$-extendible, then $V_\kappa\of\HOD$.
\end{observation}

\begin{proof} Suppose that $\kappa$ is Laver indestructibly $(\Sigma_2,0)$-extendible. Let $\Q$ be $\ltkappa$-closed forcing that forces to code every set of ordinals in $V_\kappa$ into the \GCH\ pattern at the next $\kappa$ many regular cardinals above $\kappa$. This can be done with $\ltkappa$-directed closed forcing, and so $\kappa$ remains $(\Sigma_2,0)$-extendible in the corresponding forcing extension $V[G]$, and so there is an ordinal $\theta>\kappa$ such that $V_\kappa\elesub_{\Sigma_2} V_\theta^{V[G]}$. Note that $\theta$ must be large enough that all the coding performed by $G$ is available in $V_\theta^{V[G]}$. Thus, any set $x\in V_\kappa$ is coded into the \GCH\ pattern of $V_\theta^{V[G]}$, and since this is a $\Sigma_2$ expressible property, it follows that $x$ is already coded into the \GCH\ pattern in $V_\kappa$. Thus, $x$ is ordinal definable in $V$, and so $V_\kappa\of\HOD$, as desired.
\end{proof}

\bibliographystyle{alpha}
\bibliography{MathBiblio,HamkinsBiblio}

\end{document}